\documentclass[12pt,twoside,reqno]{amsart}
\usepackage{amsaddr}
\usepackage{graphicx,bm}
\usepackage[all]{xy}   
\usepackage {amssymb,latexsym,amsthm,amsmath}
\usepackage[colorlinks, linkcolor=red, citecolor=blue]{hyperref}
\usepackage{graphicx}
\usepackage{mathtools}
\usepackage{xcolor}
\usepackage[utf8]{inputenc}
\usepackage{tikz}

\long\def\symbolfootnote[#1]#2{\begingroup%
\def\thefootnote{\fnsymbol{footnote}}\footnote[#1]{#2}\endgroup}

\newcommand{\thmref}[1]{Theorem~\ref{#1}}

\newcommand{\corref}[1]{Corollary~\ref{#1}}

\newcommand{\figref}[1]{Figure~\ref{#1}}

\makeatletter
\def\imod#1{\allowbreak\mkern10mu({\operator@font mod}\,\,#1)}
\makeatother

\newtheorem{theorem}{Theorem}[section]
\newtheorem{lemma}[theorem]{Lemma}
\newtheorem{corollary}[theorem]{Corollary}
\newtheorem{proposition}[theorem]{Proposition}
\newtheorem*{theorem*}{Theorem}

\theoremstyle{definition}
\newtheorem{remark}[theorem]{Remark}

\newtheorem{definition}[theorem]{Definition}

\numberwithin{equation}{section}

\newcommand{\ignore}[1]{}

\newcommand{\mynote}[1]{}

\def \SL {\mathrm{SL}}

\begin{document}
\setcounter{section}{0}
\setcounter{tocdepth}{1}
	% document information
\title{Groups with maximum vertex degree commuting graphs}
\author{Sushil Bhunia \and G. Arunkumar}
\address{\emph{Department of Mathematics, BITS-Pilani, Hyderabad Campus, Hyderabad, India\\
	Email: sushilbhunia@gmail.com}}
\address{\emph{Department of Mathematics, Indian Institute of Technology, Madras, India \\
Email: arun.maths123@gmail.com}}
\date{\today}
	
\begin{abstract}
Let $G$ be a finite non-abelian group and $Z(G)$ be its center. We associate a commuting graph $\Gamma(G)$ to $G$, whose vertex set is $G\setminus Z(G)$ and two distinct vertices are adjacent if they commute. 
In this paper we prove that the set of all non-abelian groups whose commuting graph has maximum vertex degree bounded above by a fixed $k \in \mathbb N$ is finite. Also, we characterize all groups for which the associated commuting graphs have maximum vertex degree at most $4$. 
\end{abstract}

\maketitle
\section{Introduction}
Let $G$ be a finite non-abelian group and $Z(G)$ be its center. The \textit{commuting graph} of $G$, denoted by $\Gamma(G)$, is the graph whose vertex set is $G\setminus Z(G)$ and two distinct vertices $x$ and $y$ are joined by an edge if $xy=yx$. 
The idea of commuting graphs can be traced back to the works of Brauer and Fowler in \cite{bf55} and Fischer in \cite{fi71} and Kegel and Gruenberg in \cite{gk75}.
The commuting graph has been studied in several contexts, for example, see \cite{akbari06, gill17, mp13, pn19, ss02}. 
Also, it would be interesting to compare our results with Bates et al.~\cite{bates07, bates03a, bates03, bates04} where they consider the commuting involution graph for sporadic, Coxeter, symmetric and special linear groups respectively. 
We now state one of our main result of this paper: 
\begin{theorem}\label{finiteness}
	Let $\mathcal{SF}(k)$ be a set of all finite non-abelian groups $G$ whose commuting graph $\Gamma(G)$ has maximum vertex degree $<k$, then the set $\mathcal{SF}(k)$ is finite for each $k \in \mathbb{N}$.
\end{theorem}
The study of maximum vertex degree of $\Gamma(G)$ can be thought of as a special case of star freeness of $\Gamma(G)$ which we will now define. 
The complete bipartite graph $\mathrm{K}_{1,k}$  is said to be a \textit{$k$-star graph}. The graph  $\Gamma(G)$ is \textit{$k$-star free} if the $k$-star graph is not an induced subgraph of $\Gamma(G)$. 
Suppose $\Gamma(G)$ has a vertex $g$ of degree $k$. Let $g_1,g_2,\dots,g_k$ be the neighbors of $g$ in $\Gamma(G)$. We observe that the vertices $g,g_1,\dots,g_k$ form a $k$-star (not necessarily induced) in $\Gamma(G)$ that's connect to the maximum vertex degree.
One motivation to study $k$-star free graphs  comes from the claw-free graphs (for $k=3$), which is one of the most extensively studied objects among $k$-stars.  
For example, see \cite{cs08, ffz97, fo11} just to name a few. 
Classifying the groups satisfying a certain condition on the associated graphs have  always been an interesting problem, for example, in \cite{mf16}, groups with triangle-free commuting conjugacy class graphs are classified. Also, in connection with star freeness, Xuanlong et al. studied star-free non-cyclic graphs of finite groups. They classified all finite non-cyclic groups whose non-cyclic graphs are $k$-star free for $3 \le k \le 6$ (see \cite{xwk}). 

In general, classifying commuting graphs which are $k$-star free is hard. So this motivates us to embark on the classification of maximum vertex degree commuting graphs with the hope that the actual $k$-star freeness could be achieved. Now, a natural question that arises in this context is that \emph{for what groups $G$ does $\Gamma(G)$ has maximum vertex degree $k$?} In an attempt to answer this question, we prove the following necessary and sufficient conditions (for $2\leq k\leq 4$):
\begin{theorem}\label{m5star}
	The commuting graph $\Gamma(G)$ has maximum vertex degree $<5$ if and only if $G$ is isomorphic to one of the following groups:  $S_3$, $D_{10}$, $A_4$, $\mathrm{GA}(1,5)$, $A_5$; $D_8$,  $Q_8$; $D_{12}$, $C_3\rtimes C_4$, $\mathrm{SL}(2,3)$; $(C_4\times C_2)\rtimes C_2$, $C_4\rtimes C_4$, $C_8\rtimes C_2$, $D_8\times C_2$,  $Q_8\times C_2$ and $D_8\circ_{C_2}C_4$, where $C_n$ denotes the cyclic group of order $n$ and $D_8\circ_{C_2}C_4$ denotes the central product of $D_8$ and $C_4$ over $C_2$.
\end{theorem}
As a corollary, we have the following result: 
\begin{corollary}\label{m4star}
	The commuting graph $\Gamma(G)$ has maximum vertex degree $<4$ if and only if $G$ is isomorphic to one of the following groups: $S_3$, $D_{10}$, $A_4$, $\mathrm{GA}(1,5)$, $A_5$; $D_8$,  $Q_8$; $D_{12}$, $C_3\rtimes C_4$, $\mathrm{SL}(2,3)$; $(C_4\times C_2)\rtimes C_2$, $C_4\rtimes C_4$, $C_8\rtimes C_2$, $D_8\times C_2$,  $Q_8\times C_2$ and $D_8\circ_{C_2}C_4$, where $C_n$ denotes the cyclic group of order $n$.
\end{corollary}
\begin{remark}
	From the above theorem and its corollary we have the following surprising fact: 
	The commuting graph $\Gamma(G)$ has maximum vertex degree $<5$ if and only if it has maximum vertex degree $<4$. We observe that, this fact already follows from Lemmas \ref{nogroup}, \ref{c4size} and \ref{c5size}.  
\end{remark}
Also, we get the complete classification of groups whose commuting graph $\Gamma(G)$ has maximum vertex degree $<3$:  
\begin{corollary}\label{clawfree}
	The commuting graph $\Gamma(G)$ has maximum vertex degree at most $2$ if and only if $G$ is isomorphic to one of the following groups: $S_3$, $A_4$, $D_8$ and $Q_8$. This result can be readily verified using the following diagram.
\end{corollary}
\begin{figure}[ht]
	\caption{Commuting graphs} \label{fig1}
	\vskip 0.2 cm
	\centering % used for centering table
	\begin{tikzpicture}
	\node at (-0.5,0) {$\Gamma(S_3):$};
	\node (a1) at (1,0) {$\bullet$};
	\node (a2) at (2,0) {$\bullet$};
	\draw[-] (1,0) -- (2,0);
	\node (a3) at (3,0) {$\bullet$};
	\node (a4) at (4,0) {$\bullet$};
	\node (a5) at (5,0) {$\bullet$};
	\end{tikzpicture}
	\begin{tikzpicture}
	\node at (-2.0,0) {$\Gamma(D_8)=\Gamma(Q_8):$};
	\node (a1) at (1,0) {$\bullet$};
	\node (a2) at (2,0) {$\bullet$};
	\draw[-] (1,0) -- (2,0);
	\node (a3) at (3,0) {$\bullet$};
	\node (a4) at (4,0) {$\bullet$};
	\draw[-] (3,0) -- (4,0);
	\node (a5) at (5,0) {$\bullet$};
	\node (a6) at (6,0) {$\bullet$};
	\draw[-] (5,0) -- (6,0);
	\end{tikzpicture}
	\begin{tikzpicture}
	\node at (-0.5,0) {$\Gamma(A_4):$};
	\node (a1) at (1,0) {$\bullet$};
	\node (a2) at (2,0) {$\bullet$};
	\draw[-] (1,0) -- (2,0);
	\node (a3) at (3,0) {$\bullet$};
	\node (a4) at (4,0) {$\bullet$};
	\draw[-] (3,0) -- (4,0);
	\node (a5) at (5,0) {$\bullet$};
	\node (a6) at (6,0) {$\bullet$};
	\draw[-] (5,0) -- (6,0);
	\node (a7) at (7,0) {$\bullet$};
	\node (a8) at (8,0) {$\bullet$};
	\draw[-] (7,0) -- (8,0);
	\node (a9) at (9,0) {$\bullet$};
	\node (a10) at (10,0) {$\bullet$};
	\node (a11) at (9.5,0.5) {$\bullet$};
	\draw[-] (9,0) -- (10,0);
	\draw[-] (9.5,0.5) -- (10,0);
	\draw[-] (9,0) -- (9.5,0.5);
	\end{tikzpicture}
\end{figure}
%%%%%%%%%%%%%%%%%%%%%%%%%%%%%%%%%%%%%%%%
\section{Preliminaries}\label{preliminary}
In this section, we fix some notations and terminologies which will be used throughout this paper. Unless otherwise specified, we will always assume that $G$ is a non-abelian finite group. The set of all non-central elements in $G$ will be denoted by $\mathcal{NC}(G)$. For a subset $A \subset G$, we denote the centralizer of $A$ in $G$ by $C_{G}(A)$. If $A = \{x\}$, we write $C_{G}(x)$ in short. 
\begin{lemma}\label{funprok}
Let $G$ be a group and $x$ be an element of $\mathcal{NC}(G)$ such that $|C_G(x)|\geq  (k+1) + |Z(G)|$, then $\Gamma(G)$ has k-star as its subgraph. In particular, $\Gamma(G)$ has maximum vertex degree $<k$ if and only if $|C_G(x)| \leq k+ |Z(G)|$ for all $x \in \mathcal{NC}(G)$.  
\end{lemma}
\begin{proof}
	From the given condition, $C_G(x)$ contains at least $k$ non-central elements, say $\{a_1,\ldots,a_{k}\}$, that are different from $x$. Then it is clear that the vertices $\{x,a_1,\ldots,a_{k}\}$ gives rise to a $k$-star subgraph in $\Gamma(G)$. 
\end{proof}
\begin{proposition}\label{funlem3}
	Let $G$ be a finite group of order $n$ with the trivial center. If $|C_G(x)| = m$ for all $x (\neq e) \in G$, then $G$ is trivial.
\end{proposition}
\begin{proof}
	By class equation, we have, $n = 1 + k\frac{n}{m}$ for some $k \in \mathbb{Z}_{+}$. Simplifying this we get $n = \frac{m}{m-k} \in \mathbb{N.}$ Since $m | n$, $n = mt = \frac{m}{m-k} \in \mathbb{N}$ for some $t \in \mathbb{N}$. This implies that $t(m-k) = 1$. Note that $t \in \mathbb{N}$ so $t = m-k = 1$. This shows that $m = n=1$ and the proof follows.
\end{proof}
\begin{proposition}\label{centersize}
Suppose $G$ is a finite group such that $\Gamma(G)$ has maximum vertex degree $<k$ then the center $|Z(G)|< k$.
\end{proposition}
\begin{proof}
	Let $x\in \mathcal{NC}(G)$. By Lemma \ref{funprok}, we have $|C_G(x)|-|Z(G)|\leq k$ as $\Gamma(G)$ has maximum vertex degree $<k$. Now, we have $|C_G(x)| > |Z(G)|$ and also, we have $|Z(G)|$ divides $|C_G(x)|$, therefore $|C_G(x)|=|Z(G)|t$ for some $t\, (\neq 1)\in \mathbb{N}$. Thus $|Z(G)|(t-1)\leq k$. Hence the result.
\end{proof}
Let us recall a classical result by E. Landau ($1903$) which will be used in the proof of the following theorem.
\begin{theorem}\cite{landau}\label{landau}
	For a given natural number $n$ there are only finitely many finite groups, up to isomorphism,  having $n$ conjugacy classes.
\end{theorem}
%%%%%%%%%%%%%%%%%%%%%%%%%%%%%%%%%%%%%%
\section{Proof of \thmref{finiteness}}\label{pfiniteness}
Let $G$ be a group such that $\Gamma(G)$ has maximum vertex degree $<k$. By Lemma  \ref{funprok}, we have $|C_G(x)| \leq k + |Z(G)|$ for all $x \in \mathcal{NC}(G)$. In view of Proposition \ref{centersize}, we have $|C_G(x)| < 2k $ for all $x \in \mathcal{NC}(G)$. 
First, we \textbf{claim} that the number of distinct centralizer sizes in $G$ is bounded.
Let $\mathcal{S}=\{n_1,n_2,\dots,n_m\}$ be the set of  distinct sizes of centralizers in $G$, given any such set we claim that $m$ is bounded by a fixed constant. Let the class equation of $G$ be 
\[|G| = \Big(\underbrace{\frac{|G|}{n_1} + \cdots + \frac{|G|}{n_1}}_{r_1}\Big) +\Big(\underbrace{\frac{|G|}{n_2} + \cdots + \frac{|G|}{n_2}}_{r_2}\Big) + \cdots + \Big(\underbrace{\frac{|G|}{n_m} + \cdots + \frac{|G|}{n_m}}_{r_m}\Big).\]
Therefore we have 
\[1 = \Big(\underbrace{\frac{1}{n_1} + \cdots + \frac{1}{n_1}}_{r_1}\Big) + \Big(\underbrace{\frac{1}{n_2} + \cdots + \frac{1}{n_2}}_{r_2}\Big) + \cdots + \Big(\underbrace{\frac{1}{n_m} + \cdots + \frac{1}{n_m}}_{r_m}\Big),\] 
It is given that $n_i< 2k$ for all $1 \le i \le m$ and hence $ \frac{1}{n_i}> \frac{1}{2k}$ for all $1 \leq i \leq m$.
We have \[1 = \frac{r_1}{n_1}+\frac{r_2}{n_2}+\cdots+\frac{r_m}{n_m}> \frac{r_1+r_2+\cdots+r_m}{2k} \geq \frac{m}{2k}.\]
Hence for a fixed $k$, $m$ cannot be arbitrarily large and we assume this number is bounded by $s=2k$.
This shows that the cardinality of the set 
\[\mathcal{C}(k) = \{\{n_1,n_2,\ldots,n_m\} \mid  n_1<n_2<\cdots<n_m \},\]
is finite as the set $\mathcal{S}$ can be identified with a subset of $[1,2k]^s$, where $n_1,n_2,\ldots,n_m$ are the distinct sizes of the centralizers of $G$. Again, by Proposition \ref{centersize}, we have $|Z(G)|< k$ as $\Gamma(G)$ has maximum vertex degree $<k$. Thus the claim.

Let $\mathbf n = \{n_1,n_2,\dots,n_m\}$ be a set of positive integers satisfying $n_1<n_2<\cdots<n_m$ for some $m$. We will show that there are finitely many groups $G$ for which $\mathbf n \in \mathcal{C}(k)$ . Let $\mathcal{G}(\mathbf n)$ denote the set of all groups for which $\mathbf n \in \mathcal{C}(k)$ and let $G \in \mathcal{G}(\mathbf n)$.  
Let $\mathcal{S}(\mathbf n)$ be the set of all positive integer solutions to the following equation
\[1 = \frac{x_1}{n_1} + \frac{x_2}{n_2} + \cdots + \frac{x_m}{n_m}.\] 
The above discussion shows that every $G \in \mathcal{G}(\mathbf n)$ gives rise to a solution of the above equation. Hence we have a well-defined map $\chi : \mathcal{G}(\mathbf n) \rightarrow \mathcal{S}(\mathbf n)$. Clearly, $\mathcal{S}(\mathbf n)$ is finite as each $n_i$ cannot be arbitrarily large for all $1\le i\le m$. 

Let $\{r_1,r_2,\dots,r_m\} \in \mathcal{S}(\mathbf n)$ and consider $l = \sum_{i=1}^{m} r_i$. Then, by Theorem \ref{landau}, the set of all finite groups with $l$ number of conjugacy classes is finite. Therefore for each $\{r_1,r_2,\dots,r_m\} \in \mathcal{S}(\mathbf n)$ has finitely many pre-images (can be empty also). i.e., all the fibers of the map $\chi$ are finite which proves that the set $\mathcal{G}(\textbf{n})$ is finite.

This completes the proof. \qed
%%%%%%%%%%%%%%%%%%%%%%%%%%%%%%%%%%%%%%%%%%%%%%%%
\section{Star free commuting graphs}\label{starfreegraphs}
In this section, we give a complete classification of groups $G$ whose commuting graphs $\Gamma(G)$ has maximum vertex degree is $k$ for $1\leq k \leq 4$. 
The classification is given by considering the following two cases separately.
\subsection{Trivial center case}
In this case, by Lemma \ref{funprok} we have, $\Gamma(G)$ has maximum vertex degree $<5$ if and only if $|C_G(x)| \le 6$ for all $x \,(\ne e) \in G$. We shall deal case by case.
\begin{remark}
	Let $G$ be a group with trivial center. If 
	$|C_G(x)| = 2, 3, 4, 5\text{ or } 6$ for all  $x(\ne e)\in G,$  
	then by Proposition \ref{funlem3}, $G$ is the trivial group and hence $\Gamma(G)$ has maximum vertex degree $<5$. Thus, we don't need to consider these cases.
\end{remark}
%\newpage
Therefore we remain with the following cases:
\begin{align}\label{5t23}
|C_G(x)| &= 2 \text{ or } 3 \, \forall\,  x \, (\ne e)\, \in G.\\
\label{5t24}
|C_G(x)| &= 2 \text{ or } 4 \, \forall\,  x \, (\ne e)\, \in G.\\
\label{5t25}
|C_G(x)| &= 2 \text{ or } 5 \, \forall\,  x \, (\ne e)\, \in G.\\
\label{5t26}
|C_G(x)| &= 2 \text{ or } 6 \, \forall\,  x \, (\ne e)\, \in G.\\
\label{5t34}
|C_G(x)| &= 3 \text{ or } 4 \, \forall\,  x \, (\ne e)\, \in G.\\
\label{5t35}
|C_G(x)| &= 3 \text{ or } 5 \, \forall\,  x \, (\ne e)\, \in G.\\
\label{5t36}
|C_G(x)| &= 3 \text{ or } 6 \, \forall\,  x \, (\ne e)\, \in G.\\
\label{5t45}
|C_G(x)| &= 4 \text{ or } 5 \, \forall\,  x \, (\ne e)\, \in G.\\
\label{5t46}
|C_G(x)| &= 4 \text{ or } 6 \, \forall\,  x \, (\ne e)\, \in G.\\
%\end{align}
%\begin{align}
\label{5t56}
|C_G(x)| &= 5 \text{ or } 6 \, \forall\,  x \, (\ne e)\, \in G.\\
\label{5t234}
|C_G(x)|& = 2 \text{ or } 3 \text{ or } 4 \, \forall\,  x \, (\ne e)\, \in G.\\
\label{5t235}
|C_G(x)| &= 2 \text{ or } 3 \text{ or } 5 \, \forall\,  x \, (\ne e)\, \in G.\\
\label{5t236}
|C_G(x)|& = 2 \text{ or } 3 \text{ or } 6 \, \forall\,  x \, (\ne e)\, \in G.\\
\label{5t245}
|C_G(x)| &= 2 \text{ or } 4 \text{ or } 5 \, \forall\,  x \, (\ne e)\, \in G.\\
\label{5t246}
|C_G(x)| &= 2 \text{ or } 4 \text{ or } 6 \, \forall\,  x \, (\ne e)\, \in G.\\
\label{5t256}
|C_G(x)| &= 2 \text{ or } 5 \text{ or } 6 \, \forall\,  x \, (\ne e)\, \in G.\\
\label{5t345}
|C_G(x)|& = 3 \text{ or } 4 \text{ or } 5 \, \forall\,  x \, (\ne e)\, \in G.\\
\label{5t346}
|C_G(x)| &= 3 \text{ or } 4 \text{ or } 6 \, \forall\,  x \, (\ne e)\, \in G.\\
\label{5t356}
|C_G(x)| &= 3 \text{ or } 5 \text{ or } 6 \, \forall\,  x \, (\ne e)\, \in G.\\
\label{5t456}
|C_G(x)| &= 4 \text{ or } 5 \text{ or } 6 \, \forall\,  x \, (\ne e)\, \in G.\\
\label{5t2345}
|C_G(x)|& = 2 \text{ or } 3 \text{ or } 4 \text{ or } 5\, \forall\,  x \, (\ne e)\, \in G.\\
\label{5t2346}
|C_G(x)| &= 2 \text{ or } 3 \text{ or } 4 \text{ or } 6\, \forall\,  x \, (\ne e)\, \in G.\\
\label{5t2356}
|C_G(x)| &= 2 \text{ or } 3 \text{ or } 5 \text{ or } 6\, \forall\,  x \, (\ne e)\, \in G.\\
\label{5t2456}
|C_G(x)|& = 2 \text{ or } 4 \text{ or } 5 \text{ or } 6\, \forall\,  x \, (\ne e)\, \in G.\\
\label{5t3456}
|C_G(x)| &= 3 \text{ or } 4 \text{ or } 5 \text{ or } 6\, \forall\,  x \, (\ne e)\, \in G.\\
\label{5t23456}
|C_G(x)| &= 2 \text{ or } 3 \text{ or } 4 \text{ or } 5 \text{ or } 6\, \forall\,  x \, (\ne e)\, \in G.
\end{align}

\begin{lemma}\label{nogroup}
	There is no non-abelian group $G$ of order $n$ with trivial center which satisfies Equations \eqref{5t24}, \eqref{5t26}, \eqref{5t35}, \eqref{5t36}, \eqref{5t46}, \eqref{5t56}, \eqref{5t234}, \eqref{5t235}, \eqref{5t236}, \eqref{5t245}, \eqref{5t246}, \eqref{5t256}, \eqref{5t346}, \eqref{5t356}, 
	\eqref{5t456}, \eqref{5t2345}, \eqref{5t2346}, \eqref{5t2356}, \eqref{5t2456}, \eqref{5t3456}, and \eqref{5t23456}.
\end{lemma}
\begin{proof}
	The proof of this lemma follows from the class equation. Here we are going to prove for two equations to see how it works. Also, this simple (and elegant) but powerful argument has been used several times in our paper. 
	From Equation \eqref{5t24}, we have $n-1 = k_1 \frac{n}{2} + k_2 \frac{n}{4}$. Clearly, $k_1 = 1$, hence we have, $\frac{n-2}{2} = \frac{nk_2}{4}$. Therefore $n(2-k_2) = 4$ with $k_2 \in \mathbb{N}$. Hence $n = 4$ and this case is not possible.

	From Equation \eqref{5t234}, we have $n-1 = k_1\frac{n}{2}+k_2\frac{n}{3}+k_3\frac{n}{4}$. Clearly, $k_1 = 1$ so we have $\frac{n-2}{2} = k_2\frac{n}{3}+k_3\frac{n}{4}$. Note that $1 \le k_2 \le 2$.
	If $k_2 = 1$, then we have $\frac{n-2}{2} = \frac{n}{3}+k_3\frac{n}{4}$. Therefore $n(2-3k_3) = 12$, which is not possible as $k_3, n \in \mathbb{N}$.
	If $k_2 = 2$, then we have $\frac{n-2}{2} = 2\frac{n}{3}+k_3\frac{n}{4}$. Hence $\frac{-n-6}{6} = \frac{nk_3}{4}$, which is absurd.
	The remaining cases are similar.
\end{proof}

\begin{lemma}\label{2p23}
	Let $G$ be a finite group of order $n$ with the trivial center and satisfies Equation \eqref{5t23}, then $G$ is isomorphic to  $S_3$. 
\end{lemma}
\begin{proof}
	By class equation we have $n-1 = k_1\frac{n}{2} + k_2 \frac{n}{3}$ for some $k_1, k_2 \in \mathbb{N}$. Clearly, $k_1 = 1$ and thus we have $\frac{n}{2}-1 = k_2 \frac{n}{3}$. This implies that $n(3-2k_2) = 6$, where $n, k_2 \in \mathbb{N} $. Hence $n = 6$ and the result follows.
\end{proof}
\begin{lemma}\label{3p25}
	Let $G$ be a finite group of order $n$ with the trivial center and satisfies Equation \eqref{5t25}, then $G$ is isomorphic to $D_{10}$. 
\end{lemma}
\begin{proof}
	Let $G$ be a group of order $n$ which satisfies the given assumptions.
	In this case, we have $n-1 = k_1 \frac{n}{2} + k_2 \frac{n}{5}$. Clearly, $k_1 = 1$. Thus we have $\frac{n-2}{2} = \frac{k_2 n}{5}$. Therefore $n(5-2k_2) = 10$ where $k_2 \in \mathbb{N}$. Hence $k_2 = 2$ and $n = 10$. Now $D_{10}$ is the only non-abelian group of order $10$ along with the class equation $1+2+2+5$. Hence $G$ is isomorphic to $D_{10}$.
\end{proof}
\begin{lemma}\label{3p34}
	Let $G$ be a finite group of order $n$ with the trivial center and satisfies Equation \eqref{5t34}, then $G$ is isomorphic to $A_4$. 
\end{lemma}
\begin{proof}
	From Equation \eqref{5t34}, we have, $n-1 = k_1\frac{n}{3}+k_2\frac{n}{4}$. Clearly, $1 \le k_1 \le 2$.
	If $k_1 =1$, then we have $n-1 = \frac{n}{3}+k_2\frac{n}{4}$. Thus $n(8-3k_2) = 12$, where $n, k_2 \in \mathbb{N}$. Therefore $k_2 = 2$ and $n=6$. Hence $G \cong S_3$ which is clearly not possible.
	If $k_1 =2$, then we have $n-1 = 2\frac{n}{3}+k_2\frac{n}{4}$. Thus $n(4-3k_2) = 12$ with $n, k_2 \in \mathbb{N}$. Therefore $k_2 = 1 \text{ and } n=12$. Clearly,  $D_{12}$ and $C_3\rtimes C_4$ are not satisfying the given assumptions but $A_4$ does. Hence $G \cong A_4$.   
\end{proof}
\begin{lemma}\label{ga15}
	The general affine group of order $1$ over  $\mathbb{F}_5$, denoted by $\mathrm{GA}(1,5)$, is the only group that satisfies Equation \eqref{5t45}. 
\end{lemma}
\begin{proof}
	Follows from the class equation ($20=1+4+5+5+5$).
\end{proof}

\begin{lemma}\label{a5}
	The alternating group $A_5$ is the only group 
	which satisfies Equation \eqref{5t345}. 
\end{lemma}
\begin{proof}
	Follows from the class equation ($60=1+12+12+15+20$).	
\end{proof}
\subsection{Non-trivial center case}
In this case, by Lemma \ref{funprok}, we have $\Gamma(G)$ has maximum vertex degree $<k$ if and only if $|C_G(x)| \leq k + |Z(G)|$ for all $x \in \mathcal{NC}(G)$. Again by Proposition \ref{centersize}, we have $|Z(G)|< k$. By our assumption of this section, we have $|Z(G)| \geq 2$, i.e., $2\leq |Z(G)|<k$. Now we shall deal case by case for $k\leq 5$.
\begin{lemma}\label{c46}
Let $G$ be a group with non-trivial center such that $\Gamma(G)$ has maximum vertex degree $<3$, then $|C_G(x)|=4$ for $x$ in $\mathcal{NC}(G)$. 
\end{lemma}
\begin{proof}
	Let $x\in \mathcal{NC}(G)$ then from Lemma  \ref{funprok}, we have $|C_G(x)|-|Z(G)|\leq 3$. Now we have $|Z(G)|=2$, then $|C_G(x)|=4$, as $|Z(G)|$ divides $|C_G(x)|$ and $Z(G)<C_G(x)$.
\end{proof}

\begin{lemma}\label{4p2}
	Let $G$ be a finite group of order $n$ with non-trivial center such that $\Gamma(G)$ has maximum vertex degree $<3$, then $G$ is isomorphic to $Q_8$ or $D_8$.
\end{lemma}
\begin{proof}
	By Lemma \ref{c46}, we have  $|C_G(x)| = 4$ for all $x \in \mathcal{NC}(G)$. 
	Now by class equation,  $n-2 = k \frac{n}{4}$ for some $k \in \mathbb{N}$. This implies that $n = \frac{8}{4-k}$ for some $k \in \mathbb{N}$. Thus $n = 4$ or $8$. Since we are interested in non-abelian groups, we have $n = 8$. Note that both $Q_8$ and $D_8$ have  conjugacy class sizes $1,1,2,2,2$ respectively. Now from the assumptions, it is clear that $G \cong Q_8$ or $D_8$.
\end{proof}
\begin{lemma}\label{c4size}
	Let $G$ be a group with non-trivial center such that $\Gamma(G)$ has maximum vertex degree $<4$, then exactly one of the following holds: $|C_G(x)|=4$ or $|C_G(x)|=6$ or $|C_G(x)|=4 \text{ or } 6$ for all $x \in \mathcal{NC}(G)$.
\end{lemma}
\begin{proof}
	Let $x\in \mathcal{NC}(G)$ then from Lemma  \ref{funprok}, we have $|C_G(x)|-|Z(G)|\leq 4$. There are two  possibilities of the order of the center. If $|Z(G)|=2$, then we have $|C_G(x)|=4$ or $|C_G(x)|=6$ or both, as $|Z(G)|$ divides $|C_G(x)|$ and $Z(G)<C_G(x)$.
	Now by similar argument if $|Z(G)|=3$ then we have $|C_G(x)|=6$. 
\end{proof}
\begin{lemma}\label{c5size}
	Let $G$ be a group with non-trivial center such that $\Gamma(G)$ has maximum vertex degrre $<5$, then exactly one of the following holds: $|C_G(x)|=4$ or $|C_G(x)|=6$ or $|C_G(x)|=8$ or $|C_G(x)|=4 \text{ or } 6$ for all $x \in \mathcal{NC}(G)$.
\end{lemma}
\begin{proof}
	Let $x\in \mathcal{NC}(G)$ then we have $|C_G(x)|-|Z(G)|\leq 5$ (by Lemma  \ref{funprok}). There are three possibilities of the order of the center. If $|Z(G)|=2$, then we have $|C_G(x)|=4$ or $|C_G(x)|=6$ or both, as $|Z(G)|$ divides $|C_G(x)|$ and $Z(G)<C_G(x)$.
	Now by similar argument if $|Z(G)|=3$, then we have $|C_G(x)|=6$. Again if $|Z(G)|=4$ then by a similar argument, we have $|C_G(x)|=8$. 
\end{proof}
\begin{remark}\label{5star}
There is no group $G$ of order $n$ with $|Z(G)|=2$ and $|C_G(x)|=6$ for all $x\in \mathcal{NC}(G)$ such that $\Gamma(G)$ has maximum vertex degree $<5$ follows from the class equation.		
\end{remark}
\begin{lemma}\label{sl23}
	Let $G$ be a group of order $n$ with $|Z(G)|=2$ such that it has centralizer sizes both $4$ and $6$, then $G$ is isomorphic to $D_{12}$ or $C_3\rtimes C_4$ or    $\mathrm{SL}(2,3)$.
\end{lemma}
\begin{proof}
	The proof follows from the class equation computations (cf. Lemma \ref{c4size} and  Lemma \ref{c5size}).
\end{proof}
\begin{lemma}\label{center3}
	There is no group $G$ of order $n$ with $|Z(G)|=3$ such that $\Gamma(G)$ has maximum vertex degree $<5$.
\end{lemma}
\begin{proof}
	In view of Lemma \ref{funprok}, we have $|C_G(x)|\leq 8$ for all $x \in \mathcal{NC}(G)$. Now we have $Z(G)<C_G(x)$ and $|Z(G)|$ divides $|C_G(x)|$, therefore $|C_G(x)|=6$ for all $x\in\mathcal{NC}(G)$. 
	Now by using class equation,  $n-3 = k \frac{n}{6}$ for some $k \in \mathbb{N}$. This implies that $n = \frac{18}{6-k}$ for some $k \in \mathbb{N}$. Then $n = 6, 9$ or $18$. Since we are interested in non-abelian groups, we have $n = 6$ or $18$. Clearly, $n \neq 6$ as $S_3$ has trivial center. So we remain with the non-abelian group of order $18$ with order of the center is $3$. As we know from the classification of finite groups of order $18$ there are no such groups with this property.
	Hence there is no group satisfying the said property.
\end{proof}
\begin{lemma}\label{six}
	Let $G$ be a group of order $n$ with $|Z(G)|=4$ and $|C_G(x)|=8$ for all $x\in \mathcal{NC}(G)$, then $G$ is isomorphic to $(C_4\times C_2)\rtimes C_2$ or $C_4\rtimes C_4$ or $C_8\rtimes C_2$ or $D_8\times C_2$ or $Q_8\times C_2$, or $D_8\circ_{C_2}C_4$,  where $C_n$ is the cyclic group of order $n$ and $\circ_{C_2}$ denotes the central product over $C_2$.
\end{lemma}
\begin{proof}
	The result now follows from the similar computations as above using class equation (cf. Lemma \ref{c4size} and Lemma \ref{c5size}).
\end{proof}

\subsection{Proof of \thmref{m5star}}\label{pm5star}
Let $G$ be a group such that the commuting graph $\Gamma(G)$ have maximum vertex degree $<5$. On one hand, if $G$ has \textbf{trivial} center then there are $5$ possibilities of $G$ follows from Lemmas  \ref{2p23}, \ref{3p25}, \ref{3p34}, \ref{ga15} and \ref{a5}. 
On the other hand, if $G$ has \textbf{non-trivial} center then there are $11$ such possibilities of $G$ which is the following: 
there are only two groups $D_8$ and $Q_8$ with center of size $2$ and centralizers have order $4$ for all non-central elements, respectively such that the commuting graph has maximum vertex degree $<5$ (follows from Lemma \ref{4p2}).
There are only three groups $D_{12}$, $C_3\rtimes C_4$ and $\mathrm{SL}(2,3)$ with the center of order $2$ and centralizers have order both $4$ and $6$ for all non-central elements respectively such that the commuting graph has maximum vertex degree $<5$ (by Lemma  \ref{sl23}).
Let $G$ be a group with $|Z(G)|=4$ and $|C_G(x)|=8$ for all $x\in \mathcal{NC}(G)$, then $G$ has order $16$ and there are six of them satisfying the given assumptions and which follows from Lemma \ref{six}.

Conversely, if $G$ is one of the given group in the theorem then it's commuting graph $\Gamma(G)$ have  maximum vertex degree $<5$ (cf. \figref{fig1} and \figref{fig2}).

This completes the proof.  \qed
%%%%%%%%%%%%%%%%%%%%%%%%%%%%%%%%%%%%%%%%%%

Let $G$ be a group such that $\Gamma(G)$ has maximum vertex degree $<2$, then $\Gamma(G)$ is a disjoint union of paths of length $1$ and isolated vertices.
\begin{corollary}\label{cor16}
The only finite non-abelian groups whose commuting graphs have maximum vertex degree $<2$ are: $S_3, D_8, Q_8$.
\end{corollary}
\begin{remark}\label{4star}
\noindent
	\begin{enumerate}
		\item There is no group $G$ of order $n$ with $|Z(G)|=3$ such that $\Gamma(G)$ has maximum vertex degree $<4$ follows from Lemma \ref{center3}.
		\item There is no group $G$ of order $n$ with $|Z(G)|=2$ and $|C_G(x)|=6$ for all $x\in \mathcal{NC}(G)$ such that $\Gamma(G)$ has maximum vertex degree $<4$ follows from the class equation. 
	\end{enumerate}
\end{remark}
\subsection{Proof of \corref{m4star}}\label{pm4star}
Follows from Theorem \ref{m5star} and Remark \ref{4star}. \qed
\subsection{Proof of \corref{clawfree}}\label{pclawfree}
Follows from Lemma \ref{3p34} and Corollary \ref{cor16} (cf. Figure \ref{fig1}). \qed
%%%%%%%%%%%%%%%%%%%%%%%%%%
\begin{figure}[ht]
\caption{Commuting graphs:} \label{fig2}
	\vskip 0.2 cm
	\centering % used for centering table	
\begin{tikzpicture}
	\node at (-1.5,0) {$\Gamma(D_{12})=\Gamma(C_3\rtimes C_4):$};
	\node (a9) at (2,0) {$\bullet$};
	\node (a10) at (3,0) {$\bullet$};
	\node (a11) at (2.5,0.5) {$\bullet$};
	\draw[-] (2,0) -- (3,0);
	\draw[-] (2.5,0.5) -- (3,0);
	\draw[-] (2,0) -- (2.5,0.5);
	\node (a9) at (2.5,1) {$\bullet$};
	\draw[-] (2,0) -- (2.5,1);
	\draw[-] (3,0) -- (2.5,1);
	\draw[-] (2.5,0.5) -- (2.5,1);
	\node (a9) at (4,0) {$\bullet$};
	\node (a10) at (5,0) {$\bullet$};
	\draw[-] (4,0) -- (5,0);
	\node (a9) at (4,0.5) {$\bullet$};
	\node (a10) at (5,0.5) {$\bullet$};
	\draw[-] (4,0.5) -- (5,0.5);
	\node (a9) at (4,1) {$\bullet$};
	\node (a10) at (5,1) {$\bullet$};
	\draw[-] (4,1) -- (5,1);
\end{tikzpicture}
\begin{tikzpicture}
	\node at (-1,0) {$\Gamma(\SL(2,3)):$};
	\node (a9) at (2,0) {$\bullet$};
	\node (a10) at (3,0) {$\bullet$};
	\node (a11) at (2.5,0.5) {$\bullet$};
	\draw[-] (2,0) -- (3,0);
	\draw[-] (2.5,0.5) -- (3,0);
	\draw[-] (2,0) -- (2.5,0.5);
	\node (a9) at (2.5,1) {$\bullet$};
	\draw[-] (2,0) -- (2.5,1);
	\draw[-] (3,0) -- (2.5,1);
	\draw[-] (2.5,0.5) -- (2.5,1);
	\node (a9) at (4,0) {$\bullet$};
	\node (a10) at (5,0) {$\bullet$};
	\node (a11) at (4.5,0.5) {$\bullet$};
	\draw[-] (4,0) -- (5,0);
	\draw[-] (4.5,0.5) -- (5,0);
	\draw[-] (4,0) -- (4.5,0.5);
	\node (a9) at (4.5,1) {$\bullet$};
	\draw[-] (4,0) -- (4.5,1);
	\draw[-] (5,0) -- (4.5,1);
	\draw[-] (4.5,0.5) -- (4.5,1);
	\node (a9) at (6,0) {$\bullet$};
	\node (a10) at (7,0) {$\bullet$};
	\node (a11) at (6.5,0.5) {$\bullet$};
	\draw[-] (6,0) -- (7,0);
	\draw[-] (6.5,0.5) -- (7,0);
	\draw[-] (6,0) -- (6.5,0.5);
	\node (a9) at (6.5,1) {$\bullet$};
	\draw[-] (6,0) -- (6.5,1);
	\draw[-] (7,0) -- (6.5,1);
	\draw[-] (6.5,0.5) -- (6.5,1);
	\node (a9) at (8,0) {$\bullet$};
	\node (a10) at (9,0) {$\bullet$};
	\node (a11) at (8.5,0.5) {$\bullet$};
	\draw[-] (8,0) -- (9,0);
	\draw[-] (8.5,0.5) -- (9,0);
	\draw[-] (8,0) -- (8.5,0.5);
	\node (a9) at (8.5,1) {$\bullet$};
	\draw[-] (8,0) -- (8.5,1);
	\draw[-] (9,0) -- (8.5,1);
	\draw[-] (8.5,0.5) -- (8.5,1);
	\node (a1) at (10,0) {$\bullet$};
	\node (a2) at (11,0) {$\bullet$};
	\draw[-] (10,0) -- (11,0);
	\node (a1) at (10,0.5) {$\bullet$};
	\node (a2) at (11,0.5) {$\bullet$};
	\draw[-] (10,0.5) -- (11,0.5);
	\node (a1) at (10,1) {$\bullet$};
	\node (a2) at (11,1) {$\bullet$};
	\draw[-] (10,1) -- (11,1);
\end{tikzpicture}
\begin{tikzpicture}
	\node at (-2.9,0.5) {$\Gamma(D_8\times C_2)=\Gamma(Q_8\times C_2)=\Gamma(D_8\circ_{C_2}C_4)=$};
	\node at (-2.6,0) {$\Gamma(C_4\rtimes C_4)=\Gamma((C_2\times C_4)\rtimes C_2)=\Gamma(C_8\rtimes C_2):$};
	\node (a9) at (2,0) {$\bullet$};
	\node (a10) at (3,0) {$\bullet$};
	\node (a11) at (2.5,0.5) {$\bullet$};
	\draw[-] (2,0) -- (3,0);
	\draw[-] (2.5,0.5) -- (3,0);
	\draw[-] (2,0) -- (2.5,0.5);
	\node (a9) at (2.5,1) {$\bullet$};
	\draw[-] (2,0) -- (2.5,1);
	\draw[-] (3,0) -- (2.5,1);
	\draw[-] (2.5,0.5) -- (2.5,1);
	\node (a9) at (4,0) {$\bullet$};
	\node (a10) at (5,0) {$\bullet$};
	\node (a11) at (4.5,0.5) {$\bullet$};
	\draw[-] (4,0) -- (5,0);
	\draw[-] (4.5,0.5) -- (5,0);
	\draw[-] (4,0) -- (4.5,0.5);
	\node (a9) at (4.5,1) {$\bullet$};
	\draw[-] (4,0) -- (4.5,1);
	\draw[-] (5,0) -- (4.5,1);
	\draw[-] (4.5,0.5) -- (4.5,1);
	\node (a9) at (6,0) {$\bullet$};
	\node (a10) at (7,0) {$\bullet$};
	\node (a11) at (6.5,0.5) {$\bullet$};
	\draw[-] (6,0) -- (7,0);
	\draw[-] (6.5,0.5) -- (7,0);
	\draw[-] (6,0) -- (6.5,0.5);
	\node (a9) at (6.5,1) {$\bullet$};
	\draw[-] (6,0) -- (6.5,1);
	\draw[-] (7,0) -- (6.5,1);
	\draw[-] (6.5,0.5) -- (6.5,1);
\end{tikzpicture}
\begin{tikzpicture}
\node at (-0.5,0) {$\Gamma(\mathrm{GA}(1,5)):$};
\node (a1) at (1,0) {$\bullet$};
\node (a2) at (2,0) {$\bullet$};
\draw[-] (1,0) -- (2,0);
\node (a21) at (1.5,0.5) {$\bullet$};
\draw[-] (1.5,0.5) -- (1,0);
\draw[-] (1.5,0.5) -- (2,0);
\node (a3) at (3,0) {$\bullet$};
\node (a4) at (4,0) {$\bullet$};
\draw[-] (3,0) -- (4,0);
\node (a21) at (3.5,0.5) {$\bullet$};
\draw[-] (3.5,0.5) -- (3,0);
\draw[-] (3.5,0.5) -- (4,0);
\node (a5) at (5,0) {$\bullet$};
\node (a6) at (6,0) {$\bullet$};
\draw[-] (5,0) -- (6,0);
\node (a21) at (5.5,0.5) {$\bullet$};
\draw[-] (5.5,0.5) -- (5,0);
\draw[-] (5.5,0.5) -- (6,0);
\node (a7) at (7,0) {$\bullet$};
\node (a8) at (8,0) {$\bullet$};
\draw[-] (7,0) -- (8,0);
\node (a21) at (7.5,0.5) {$\bullet$};
\draw[-] (7.5,0.5) -- (7,0);
\draw[-] (7.5,0.5) -- (8,0);
\node (a9) at (9,0) {$\bullet$};
\node (a10) at (10,0) {$\bullet$};
\node (a11) at (9.5,0.5) {$\bullet$};
\draw[-] (9,0) -- (10,0);
\draw[-] (9.5,0.5) -- (10,0);
\draw[-] (9,0) -- (9.5,0.5);
\node (a9) at (11,0) {$\bullet$};
\node (a10) at (12,0) {$\bullet$};
\node (a11) at (11.5,0.5) {$\bullet$};
\draw[-] (11,0) -- (12,0);
\draw[-] (11.5,0.5) -- (12,0);
\draw[-] (11,0) -- (11.5,0.5);
\node (a9) at (11.5,1) {$\bullet$};
\draw[-] (11,0) -- (11.5,1);
\draw[-] (12,0) -- (11.5,1);
\draw[-] (11.5,0.5) -- (11.5,1);
\end{tikzpicture}
\begin{tikzpicture}
\node at (-2.9,0) {$\Gamma(D_{10}):$};
\node (a9) at (2,0) {$\bullet$};
\node (a10) at (3,0) {$\bullet$};
\node (a11) at (2.5,0.5) {$\bullet$};
\draw[-] (2,0) -- (3,0);
\draw[-] (2.5,0.5) -- (3,0);
\draw[-] (2,0) -- (2.5,0.5);
\node (a9) at (2.5,1) {$\bullet$};
\draw[-] (2,0) -- (2.5,1);
\draw[-] (3,0) -- (2.5,1);
\draw[-] (2.5,0.5) -- (2.5,1);
\node (a9) at (4,0) {$\bullet$};
\node (a10) at (5,0) {$\bullet$};
\draw[-] (4,0) -- (5,0);
\node (a9) at (6,0) {$\bullet$};
\node (a10) at (7,0) {$\bullet$};
\node (a10) at (8,0) {$\bullet$};
\end{tikzpicture}
\begin{tikzpicture}
\node at (-1,2) {$\Gamma(A_5):$};
\node (a9) at (2,0) {$\bullet$};
\node (a10) at (3,0) {$\bullet$};
\node (a11) at (2.5,0.5) {$\bullet$};
\draw[-] (2,0) -- (3,0);
\draw[-] (2.5,0.5) -- (3,0);
\draw[-] (2,0) -- (2.5,0.5);
\node (a9) at (2.5,1) {$\bullet$};
\draw[-] (2,0) -- (2.5,1);
\draw[-] (3,0) -- (2.5,1);
\draw[-] (2.5,0.5) -- (2.5,1);
\node (a9) at (4,0) {$\bullet$};
\node (a10) at (5,0) {$\bullet$};
\node (a11) at (4.5,0.5) {$\bullet$};
\draw[-] (4,0) -- (5,0);
\draw[-] (4.5,0.5) -- (5,0);
\draw[-] (4,0) -- (4.5,0.5);
\node (a9) at (4.5,1) {$\bullet$};
\draw[-] (4,0) -- (4.5,1);
\draw[-] (5,0) -- (4.5,1);
\draw[-] (4.5,0.5) -- (4.5,1);
\node (a9) at (6,0) {$\bullet$};
\node (a10) at (7,0) {$\bullet$};
\node (a11) at (6.5,0.5) {$\bullet$};
\draw[-] (6,0) -- (7,0);
\draw[-] (6.5,0.5) -- (7,0);
\draw[-] (6,0) -- (6.5,0.5);
\node (a9) at (6.5,1) {$\bullet$};
\draw[-] (6,0) -- (6.5,1);
\draw[-] (7,0) -- (6.5,1);
\draw[-] (6.5,0.5) -- (6.5,1);
\node (a9) at (8,0) {$\bullet$};
\node (a10) at (9,0) {$\bullet$};
\node (a11) at (8.5,0.5) {$\bullet$};
\draw[-] (8,0) -- (9,0);
\draw[-] (8.5,0.5) -- (9,0);
\draw[-] (8,0) -- (8.5,0.5);
\node (a9) at (8.5,1) {$\bullet$};
\draw[-] (8,0) -- (8.5,1);
\draw[-] (9,0) -- (8.5,1);
\draw[-] (8.5,0.5) -- (8.5,1);
\node (a9) at (10,0) {$\bullet$};
\node (a10) at (11,0) {$\bullet$};
\node (a11) at (10.5,0.5) {$\bullet$};
\draw[-] (10,0) -- (11,0);
\draw[-] (10.5,0.5) -- (11,0);
\draw[-] (10,0) -- (10.5,0.5);
\node (a9) at (10.5,1) {$\bullet$};
\draw[-] (10,0) -- (10.5,1);
\draw[-] (11,0) -- (10.5,1);
\draw[-] (10.5,0.5) -- (10.5,1);
\node (a9) at (12,0) {$\bullet$};
\node (a10) at (13,0) {$\bullet$};
\node (a11) at (12.5,0.5) {$\bullet$};
\draw[-] (12,0) -- (13,0);
\draw[-] (12.5,0.5) -- (13,0);
\draw[-] (12,0) -- (12.5,0.5);
\node (a9) at (12.5,1) {$\bullet$};
\draw[-] (12,0) -- (12.5,1);
\draw[-] (13,0) -- (12.5,1);
\draw[-] (12.5,0.5) -- (12.5,1);

	\node (a9) at (0,0) {$\bullet$};
\node (a10) at (1,0) {$\bullet$};
\draw[-] (0,0) -- (1,0);
\node (a9) at (0,0.5) {$\bullet$};
\node (a10) at (1,0.5) {$\bullet$};
\draw[-] (0,0.5) -- (1,0.5);
\node (a9) at (0,1) {$\bullet$};
\node (a10) at (1,1) {$\bullet$};
\draw[-] (0,1) -- (1,1);
\node (a9) at (0,0.25) {$\bullet$};
\node (a10) at (1,0.25) {$\bullet$};
\draw[-] (0,0.25) -- (1,0.25);
\node (a9) at (0,0.75) {$\bullet$};
\node (a10) at (1,0.75) {$\bullet$};
\draw[-] (0,0.75) -- (1,0.75);

\node (a1) at (1,2) {$\bullet$};
\node (a2) at (2,2) {$\bullet$};
\draw[-] (1,2) -- (2,2);
\node (a21) at (1.5,2.5) {$\bullet$};
\draw[-] (1.5,2.5) -- (1,2);
\draw[-] (1.5,2.5) -- (2,2);
\node (a3) at (3,2) {$\bullet$};
\node (a4) at (4,2) {$\bullet$};
\draw[-] (3,2) -- (4,2);
\node (a21) at (3.5,2.5) {$\bullet$};
\draw[-] (3.5,2.5) -- (3,2);
\draw[-] (3.5,2.5) -- (4,2);
\node (a5) at (5,2) {$\bullet$};
\node (a6) at (6,2) {$\bullet$};
\draw[-] (5,2) -- (6,2);
\node (a21) at (5.5,2.5) {$\bullet$};
\draw[-] (5.5,2.5) -- (5,2);
\draw[-] (5.5,2.5) -- (6,2);
\node (a7) at (7,2) {$\bullet$};
\node (a8) at (8,2) {$\bullet$};
\draw[-] (7,2) -- (8,2);
\node (a21) at (7.5,2.5) {$\bullet$};
\draw[-] (7.5,2.5) -- (7,2);
\draw[-] (7.5,2.5) -- (8,2);
\node (a9) at (9,2) {$\bullet$};
\node (a10) at (10,2) {$\bullet$};
\node (a11) at (9.5,2.5) {$\bullet$};
\draw[-] (9,2) -- (10,2);
\draw[-] (9.5,2.5) -- (10,2);
\draw[-] (9,2) -- (9.5,2.5);

	\node (a9) at (11,2) {$\bullet$};
\node (a10) at (12,2) {$\bullet$};
\draw[-] (11,2) -- (12,2);
\node (a9) at (11,2.5) {$\bullet$};
\node (a10) at (12,2.5) {$\bullet$};
\draw[-] (11,2.5) -- (12,2.5);
\node (a9) at (11,3) {$\bullet$};
\node (a10) at (12,3) {$\bullet$};
\draw[-] (11,3) -- (12,3); 
\node (a9) at (11,2.25) {$\bullet$};
\node (a10) at (12,2.25) {$\bullet$};
\draw[-] (11,2.25) -- (12,2.25);
\node (a9) at (11,2.75) {$\bullet$};
\node (a10) at (12,2.75) {$\bullet$};
\draw[-] (11,2.75) -- (12,2.75);

\end{tikzpicture}	
\end{figure}

%%%%%%%%%%%%%%%%%%%%%%%%%%%%%%%%%%
\section{Dihedral groups and star-freeness}\label{d2n}
In this section, we give an example of a group, namely, dihedral group $D_{2n}$ of order $2n$ and its star-freeness.
Every finite non-abelian group $G$ belongs to $\mathcal{SF}(k)$ for some $k$.
We have $\mathcal{SF}(1) \subseteq \cdots \subseteq \mathcal{SF}(k) \subseteq \cdots$.
We show that all these inclusions are strict (Propositions \ref{prop1} and \ref{prop2}).
Define $\mathcal{SF}$ to be the limit of the above inclusions. Clearly, $\mathcal{SF}$ is equal to the set of all non-abelian groups. 
\begin{definition}(\textbf{Vertex degree type $(k)$})
	A group $G$ is said to be of \emph{vertex degree type ($k$)} if the commuting graph $\Gamma(G)$ has maximum vertex degree $<k$.
\end{definition}
If $G$ is a group of vertex degree type ($k$) then $G$ is vertex degree type ($k+1$). In view of this simple observation, we define the \textit{star number}, denoted by $\mathrm{S}(G)$, to be the smallest $k$ such that $G$ is of vertex degree type ($k$).	
\subsection{Trivial center case}
If $G$ has trivial center then we have, $G \in \mathcal{SF}(k)$ if and only if  $|C_G(x)|<(k+1)+1$ for all $x \in \mathcal{NC}(G)$. 
More precisely, the size of the largest non-central centralizer of $G$ should be strictly less than $k+2$.

In this case, $n$ is odd and the class equation of $D_{2n}$ is given by 
$$2n-1 = \Big(\frac{n-1}{2}\Big) \frac{2n}{n} + \Big(1\Big) \frac{2n}{2}.$$
So the size of the largest centralizer in $D_{2n}$ is $n$ and hence 
$D_{2n} \in \mathcal{SF}(k)$ if and only if $k > n - 2$.
We record this as:
\begin{proposition}\label{prop1}
	Assume $n$ is odd, then $D_{2n} \in \mathcal{SF}(k)$ for all $k \geq n-1$ and $S(D_{2n}) = n-1$.
\end{proposition}
\begin{remark}
	Let $k_1$ and $k_2$ be two consecutive odd numbers, then $D_{2k_2} \in \mathcal{SF}(k_1+1)$ and $D_{2k_2} \notin \mathcal{SF}(k_1)$.
\end{remark}
\subsection{Non-trivial center case}
If $G$ has non-trivial center then we have, $G \in \mathcal{SF}(k)$ if and only if $|C_G(x)|<(k+1)+|Z(G)|$ for all $x \in \mathcal{NC}(G)$. 
More precisely, the size of the largest non-central centralizer of $G$ should be strictly less than $k+1+|Z(G)|$.

In this case, $n$ is even and the class equation of $D_{2n}$ is given by 
$$2n-2 = \Big(\frac{n-2}{2}\Big) \frac{2n}{n} + \Big(2\Big) \frac{2n}{4}.$$
So the size of the largest centralizer in $D_{2n}$ is $n$ and $|Z(G)| = 2$. Hence 
$D_{2n} \in \mathcal{SF}(k)$ if and only if $k > n - 3$.
We record this as:
\begin{proposition}\label{prop2}
	Assume $n$ is even, then $D_{2n} \in \mathcal{SF}(k)$ for all $k \geq n-2$ and $S(D_{2n}) = n-2$.
\end{proposition}
\begin{remark}
	Let $k_1$ and $k_2$ be two consecutive even numbers, then $D_{2k_2} \in \mathcal{SF}(k_1)$ and $D_{2k_2} \notin \mathcal{SF}(k_1-1)$.
\end{remark}

\medskip
\noindent
\textbf{Acknowledgement:} The authors would like to thank the anonymous referee for carefully reading
this paper.
%%%%%%%%%%%%%%%%%%%%%%%%%%%%%%%%%%%%%%%%%%%%%%%%%%%%%%%%%%%%%

\end{document}